\documentclass[preprint,5p,times,twocolumn]{elsarticle}

\usepackage{graphicx}
\usepackage[misc,geometry]{ifsym}
\usepackage{blindtext}
\usepackage[title]{appendix}
\usepackage{amssymb}
\usepackage{comment}

\usepackage{enumitem}
\usepackage{bbm}
\usepackage{algorithm}
\usepackage{algorithmic}
\usepackage{placeins}
\usepackage{setspace}
\usepackage{amsmath}
\usepackage{amsfonts}

\usepackage{multirow}

\usepackage{amsthm}

\usepackage{mathrsfs}
\usepackage{tikz}

\usetikzlibrary{arrows,automata}
\usetikzlibrary{decorations.text}

\usepackage{booktabs}

\usepackage{esvect}

\usepackage{array}

\usepackage[utf8]{inputenc}

\usepackage{graphicx}
\usepackage{placeins}

\usepackage{enumitem}
\setlist{nolistsep}
\usepackage{graphicx}
\usepackage{amssymb}
 \usepackage{amsthm}
\usepackage{amsmath}
\usepackage{booktabs}
\usepackage[hidelinks]{hyperref}
\usepackage{bm}
\usepackage{mathtools}
\usepackage{pgfplots}
\usepackage{float}
\usepackage{subcaption}
\newtheorem{theorem}{Theorem}

\newtheorem{proposition}{Proposition}

\newtheorem{remark}{Remark}
\newtheorem{assumption}{Assumption}
\usepackage{lineno}
\setlist[itemize]{leftmargin=*}
\usepackage[hidelinks]{hyperref}

\journal{Operations Research Letters}

\begin{document}

\begin{frontmatter}

\title{\large On Polyhedral and Second-Order Cone Decompositions of Semidefinite Optimization Problems}

\author[label1]{Dimitris Bertsimas\corref{cor1}}
\address[label1]{Sloan School of Management and Operations Research Center, Massachusetts Institute of Technology Cambridge, MA, USA}
\address[label2]{Operations Research Center, Massachusetts Institute of Technology Cambridge, MA, USA}
\cortext[cor1]{Corresponding author}
\ead{dbertsim@mit.edu}

\author[label2]{Ryan Cory-Wright}
\ead{ryancw@mit.edu}

\begin{abstract}
We study a cutting-plane method for semidefinite optimization problems,
and supply a proof of the method's convergence, under a boundedness assumption. By relating the method's rate of convergence to an initial outer approximation's diameter, we argue  the method performs well when initialized with a second-order cone approximation, instead of a linear approximation. We invoke the method to provide bound gaps of $0.5$-$6.5\%$ for sparse PCA problems with $1000$s of covariates, and solve nuclear norm problems over $500 \times 500$ matrices.
\end{abstract}

\begin{keyword}
Semidefinite optimization \sep Cutting plane \sep Nuclear norm \sep Sparse PCA.
\end{keyword}

\end{frontmatter}

\section{Introduction}
We study decomposition schemes for semidefinite optimization problems (SDOs) of the form:
\begin{equation}\label{primalsdp}
\begin{aligned}
    \min_{\bm{X} \in \mathbb{R}^{n \times n}} \quad & \left\langle \bm{C}, \bm{X} \right\rangle \\
    \text{s.t.} \quad & \left\langle \bm{A}_i, \bm{X} \right\rangle =b_i, \quad \forall i \in [m],\quad \bm{X} \succeq \bm{0}.
\end{aligned}
\end{equation}

{\color{black} Problem \eqref{primalsdp} is tractable from a traditional complexity theory perspective, because we can solve it
%usually\footnote{In general, Problem \eqref{primalsdp} is not solvable in polynomial time, because it may admit solutions of exponential size \citep[see][Section 3.3]{alizadeh1995interior}. Nonetheless, by bounding the allowable size of solutions to Problem \eqref{primalsdp} apriori, it can be solved in polynomial time \cite{alizadeh1995interior}. As most practically relevant SDOs can be bounded in this fashion, we refer to SDOs as ``polynomial time solvable''; see \citet{ramana1997exact} for a complete characterization of the complexity of SDOs.}
to $\epsilon$-optimality in polynomial time
via interior point methods (IPMs); see \citet{ramana1997exact} for a complete characterization of the complexity of SDOs.} However, it is notoriously difficult to solve in practice, because IPMs memory requirements scale at a demanding rate. Indeed, state-of-the-art SDO solvers such as \verb|MOSEK| cannot solve constrained instances of Problem \eqref{primalsdp} with $n>250$ variables on a standard laptop, and it is optimization folklore that there is a gap between SDOs theoretical and practical tractability.

Motivated by the demanding memory requirements of IPMs, a stream of literature studies \textit{inexact} methods for SDOs, which replace the semidefinite constraint with weaker yet less computationally demanding constraints. This approach was first investigated by \citet{kim2001second}, who observed that relaxing a positive semidefinite constraint to the weaker constraint that all $2 \times 2$ minors of a matrix are positive semidefinite yields a second-order cone (SOC)-representable outer approximation of the positive semidefinite (PSD) cone.

In a related line of work, \citet{krishnan2006unifying} propose applying \citet{kelley1960cutting}'s cutting plane method to generate an improving sequence of outer approximations of the positive semidefinite cone. Their approach converges whenever the feasible region is bounded, albeit exponentially slowly in the worst-case. %{\color{black}(although this re
%to our knowledge, no proof of convergence has been supplied for SDOs. Therefore, we supply a proof in Theorem \ref{proofofconvergence})}
Unfortunately, in \cite{krishnan2006unifying}'s implementation, Kelley's method performs poorly in practice and often provides weaker bounds than those obtained by enforcing SOC-representable $2\times 2$ minor constraints, as found empirically in \cite{sivaramakrishnan2002linear, ahmadi2014dsos}.

From a traditional complexity theory perspective, the slow rate of convergence of \citet{kelley1960cutting}'s method is unsurprising. Indeed, SDOs cannot be approximated arbitrarily well using polynomially-sized linear optimization problems (LOs) \citep[Proposition 3]{braun2015approximation}, and therefore any LO-based cutting-plane method for SDOs converges exponentially slowly in the worst-case. However, there are numerous examples of cutting-plane methods with slow worst-case rates of convergence which nonetheless successfully solve large-scale problems in practice \citep[see][and references therein]{dantzig1954solution, fischetti2016redesigning, bertsimas2019unified}. Moreover, optimizers widely believe that there is a gap between the observed and worst-case performance of cutting-plane methods \citep[see][Section 5.2]{mutapcic2009cutting}. Therefore, a slow worst-case rate of convergence does not preclude Kelley's method from performing well in practice.

\subsection{Our Contributions}
In this paper, we provide a different perspective on Kelley's method than the traditional complexity theory perspective. By bounding the number of iterations required to obtain an $\epsilon$-feasible solution to Problem \eqref{primalsdp}, we establish that the method's worst-case performance depends explicitly on the diameter of an initial outer approximation of the problem, and argue that this dependence also appears in practice.

By examining this worst-case bound, we diagnose why existing implementations of \citet{kelley1960cutting}'s method do not scale well in practice. Namely, they are initialized with LO-representable outer approximations of the PSD cone, which do not capture its inherent nonlinearity. Indeed, as we establish in Section \ref{sec:socpaprox}, the diameter of an LO-representable approximation grows linearly with $n \cdot \mathrm{tr}(\bm{X})$, and therefore the rate of convergence of Kelley's method (under this initial approximation) depends upon $n \cdot \mathrm{tr}(\bm{X})$. After observing that Kelley's method performs poorly in both theory and practice with an LO-representable initial approximation, we propose initializing Kelley's method with a SOC-representable approximation. We establish that (a) the diameter of this approximation grows at the slower rate of $ \mathrm{tr}(\bm{X})$, and (b) initializing a cutting-plane method with a SOC approximation performs better in both theory and practice.

Finally, we apply Kelley's method to two problems from the machine learning literature: sparse principal component analysis (PCA) and nuclear norm minimization. For sparse PCA, we demonstrate that initializing Kelley's method with an SOC-representable feasible region provides bounds which are comparable with the SDO-representable bound of \citet{d2005direct}, and successfully supplies near-exact bounds for problems with $1000$s of covariates. For nuclear norm minimization {\color{black}problems}, we reformulate the nuclear norm in terms of its dual norm, and invoke this reformulation to derive a new class of cutting-planes. Moreover, we demonstrate that combining Kelley's method with this class of cuts scales better than IPMs for nuclear norm minimization {\color{black}problems}.

\subsection{Structure}
The rest of this paper is laid out as follows:
\begin{itemize}
    \item In Section \ref{sec:socpaprox}, we study two outer-approximations of Problem \eqref{primalsdp}: a linear approximation and a second-order cone approximation. We also supply bounds on the distance between these approximations and the original problem.
    \item In Section \ref{sec:outeraprox}, we present an outer-approximation method for solving Problem \eqref{primalsdp} to certifiable optimality, and supply a proof of its convergence. Furthermore, we establish that the worst-case rate of convergence depends explicitly on the diameter of the initial feasible region.
    \item In Section \ref{sec:numres}, we present numerical results demonstrating that the outer-approximation method is competitive with state-of-the-art methods for both sparse principal component analysis and nuclear norm minimization problems.
\end{itemize}

\subsection{Notation}
We let nonbold face characters such as $b$ denote scalars, lowercase bold faced characters such as $\bm{x}$ denote vectors, uppercase bold faced characters such as $\bm{X}$ denote matrices, and calligraphic uppercase characters such as $\mathcal{Z}$ denote sets. We assume that matrices $\bm{X}$ are symmetric, unless stated otherwise. We let $[n]$ denote the set of running indices $\{1, \ldots, n\}$. We let $\mathbf{e}$ denote a vector of all $1$'s, $\bm{0}$ denote a vector of all $0$'s, and $\mathbb{I}$ denote the identity matrix, with dimension implied by the context. We let $S_+^n$ denote the $n \times n$ positive semidefinite cone.

We also use an assortment of matrix operators. We let $\Vert \cdot \Vert_{*}$ denote the spectral norm of a matrix, $\langle \cdot,\cdot \rangle$ denote the Euclidean inner product between two matrices, $\Vert \cdot \Vert_F$ denote the Frobenius norm of a matrix, and $\Vert \cdot \Vert_\sigma$ denote the singular value norm of a matrix; see \citet{horn1990matrix} for a general theory of matrix operators.

Finally, we let $D(\mathcal{X})$ denote the diameter of a set $\mathcal{X}$, i.e.,
$$D(\mathcal{X})=\sup_{\bm{x}, \bm{y} \in \mathcal{X}} \left\Vert \bm{x}-\bm{y}\right\Vert,$$ where $\Vert \cdot \Vert$ is the $l_2$ (Frobenius) norm for vectors (matrices).

\section{Two Outer Approximations of the PSD cone}\label{sec:socpaprox}
In this section, we introduce two outer approximations of the positive semidefinite cone, bound the diameters of the approximations, and propose two classes of separating hyperplanes which iteratively refine the approximations.

\subsection{A Linear Outer Approximation}
The following problem is a valid relaxation of Problem \eqref{primalsdp}, which is LO-representable \citep[see][Table 1]{permenter2014partial}:
\begin{equation}
\begin{aligned}\label{ddsos}
    \min_{\bm{X} \in S^n} \quad & \left\langle \bm{C}, \bm{X} \right\rangle &\\
    \text{s.t.} \quad & \left\langle \bm{A}_i, \bm{X} \right\rangle =b_i, & \forall i \in [m],\\
    & X_{i,i} \geq 0, & \forall i \in [n],\\
    & X_{i,i}+X_{j,j} \pm 2 X_{i,j} \geq 0, & \forall i \in [n], \forall j \in [n].
\end{aligned}
\end{equation}

In Section \ref{sec:outeraprox}, we will iteratively improve the quality of outer approximations of SDOs, such as Problem \eqref{ddsos}, by adding cutting planes, and prove that any limit point of this procedure solves Problem \eqref{primalsdp}. As Section \ref{sec:outeraprox}'s proof of convergence relies on a compactness argument, which in turn rests on the boundedness of the initial feasible region, we now prove that imposing the constraint $\mathrm{tr}(\bm{X}) \leq T$ in Problem \eqref{ddsos} ensures that all feasible solutions are contained within a convex, compact set.

\begin{proposition}\label{l1normboundedness}
Let there exist a $T \geq 0$ such that all solutions to Problem \eqref{primalsdp} are bounded in trace norm by $T$, i.e., $\mathrm{tr}(\bm{X}) \leq T$. Then, the solutions to Problem \eqref{ddsos} form a convex, compact set.
\end{proposition}
\begin{proof}
The boundedness of the set follows because:
\begin{equation*}
\begin{aligned}
    \sum_{i=1}^n \sum_{j=1}^n \left\vert X_{i,j}\right\vert \leq \sum_{i=1}^n \sum_{j=1}^n \frac{1}{2}(X_{i,i}+X_{j,j})\leq n \cdot \mathrm{tr}(\bm{X}) \leq nT.
\end{aligned}
\end{equation*}
Therefore, Problem \eqref{ddsos}'s feasible region is a polytope (i.e., a bounded polyhedron) and the result holds. \qedhere
\end{proof}
A corollary of Proposition \ref{l1normboundedness} is that $\Vert \bm{X}\Vert_F \leq nT$, because ${\Vert \cdot \Vert_2 \leq \Vert \cdot \Vert_1}$. Consequently, the set of feasible solutions to Problem \eqref{ddsos} has a diameter of $D(\mathcal{X}) \leq 2 n T$.

Unfortunately, while Problem \eqref{ddsos} is a very tractable outer approximation, various authors \cite{sivaramakrishnan2002linear, ahmadi2014dsos} have observed that it poorly approximates Problem \eqref{primalsdp} in practice. Motivated by these observations, we now argue that Problem \eqref{ddsos} should \textit{not} be used as an initial outer-approximation for a cutting-plane method, because it is dominated by an equally tractable yet much tighter nonlinear approximation. To this end, we now bound the degree to which solutions to Problem \eqref{ddsos} violate the constraint $\bm{X} \succeq \bm{0}$, in terms of $\mathrm{tr}(\bm{X})$, and demonstrate that \eqref{ddsos} offers poor constraint violation guarantees, since the worst-case value of $\lambda_{\min}(\bm{X})$ grows in magnitude with $n$: %is more negative in Problem \eqref{ddsos} than \eqref{primalsocprelaxation}:
\begin{proposition}\label{propconsviol1}
Let $\bm{X} \in \mathbb{R}^{n \times n}$ be a feasible solution to Problem \eqref{ddsos} such that $\mathrm{tr}(\bm{X}) = T$ is held constant. Then,
\begin{align*}
    \lambda_{\min}(\bm{X}) \geq\frac{T}{n}-\frac{T\sqrt{(n^3-1)(n-1)}}{n} \geq \frac{T}{n}-n T.
\end{align*}
\end{proposition}

\begin{proof}
By Proposition \ref{l1normboundedness}, we have that $\mathrm{tr}(\bm{X}^2) \leq n^2 T^2$. Therefore, the result follows from the eigenvalue-trace bound \citep[see][Theorem 2.1]{wolkowicz1980bounds}, which requires that for any matrix $\bm{X} \in S^n$
\begin{align*}
    \lambda_{\min}(\bm{X}) \geq \frac{\mathrm{tr}(\bm{X})}{n}-\sqrt{(n-1)\left(\frac{\mathrm{tr}(\bm{X}^2)}{n}-\frac{\mathrm{tr}(\bm{X})^2}{n^2}\right)}.
\end{align*}Note that this inequality is tight in the absence of further problem structure \citep[see][Theorem 2.1]{wolkowicz1980bounds}.
\end{proof}
Proposition \ref{propconsviol1} demonstrates that \textit{even if} the optimal choice of $\bm{X}$ is bounded in trace norm by a constant, the quality of Problem \eqref{ddsos}'s approximation degrades with $n$. This result explains why various authors \cite{sivaramakrishnan2002linear, kocuk2016strong, ahmadi2014dsos} have found that Problem \eqref{ddsos} poorly approximates Problem \eqref{primalsdp} in a high-dimensional setting.

We now present a SOC-representable outer approximation of Problem \eqref{primalsdp}, and establish that it enjoys more attractive worst-case guarantees than a linear outer approximation.
\subsection{A Second-Order Cone Outer Approximation}
The following problem is a valid relaxation of Problem \eqref{primalsdp}, which is SOC representable \citep[see][]{kim2001second}:
\begin{equation}
    \begin{aligned}\label{primalsocprelaxation}
        \min_{\bm{X} \in S^n} \quad & \left\langle \bm{C}, \bm{X} \right\rangle &\\
    \text{s.t.} \quad & \left\langle \bm{A}_i, \bm{X} \right\rangle =b_i, & \forall i \in [m],\\
     & \left\Vert \begin{pmatrix}
                2X_{i,j}\\
                X_{i,i}-X_{j,j}
                \end{pmatrix} \right\Vert_2 \leq X_{i,i}+X_{j,j}, & \forall i \in [n], \ \forall j \in [n].
    \end{aligned}
\end{equation}
We now relate the diameter of Problem \eqref{primalsocprelaxation}'s feasible region to a trace constraint:
\begin{proposition}\label{socpboundedness}
Let there exists a $T \geq 0$ such that all feasible solutions to Problem \eqref{primalsdp} are bounded in trace norm by $T$, i.e., $\mathrm{tr}(\bm{X}) \leq T, \forall \bm{X} \in \mathcal{X}$. Then, Problem \eqref{primalsocprelaxation}'s feasible region is bounded in Frobenius norm by $T$, i.e., $\Vert\bm{X}\Vert_F \leq T, \forall \bm{X} \in \mathcal{X}$.
\end{proposition}
\begin{proof}
The $2\times 2$ minor constraints in Problem \eqref{primalsocprelaxation} are equivalent to requiring that $X_{i,j}^2 \leq X_{i,i} X_{j,j}$. Therefore, this result follows directly from the observation that:
\begin{equation*}
\begin{aligned}
\hspace{10mm}\Vert \bm{X}\Vert_F^2=\sum_{i=1}^n \sum_{j=1}^n X_{i,j}^2 \leq \sum_{i=1}^n \sum_{j=1}^n X_{i,i}X_{j,j}= \mathrm{tr}(\bm{X})^2 \leq T^2. \hspace{0.3mm}\qedhere
\end{aligned}
\end{equation*}
\end{proof}
% Proposition \ref{socpboundedness} implies that if solutions to Problem \eqref{primalsocprelaxation} are bounded in trace norm by $T$, they are contained within a set with diameter $2T$.

We now bound the degree to which solutions to Problem \eqref{primalsocprelaxation} violate the constraint $\bm{X} \succeq \bm{0}$, and establish that Problem \eqref{primalsocprelaxation} enjoys better constraint violation guarantees than Problem \eqref{ddsos}:

\begin{proposition}\label{propviol2}
Let $\bm{X} \in S^n$ be a feasible solution to Problem \eqref{primalsocprelaxation} such that $\mathrm{tr}(\bm{X})=T$. Then
\begin{align*}
    \lambda_{\min}(\bm{X}) \geq\frac{2T}{n}-T.
\end{align*}
\end{proposition}
\begin{proof}
The result follows from observing that $\mathrm{tr}(\bm{X}^2) \leq T^2$, and invoking the eigenvalue-trace bound. The bound is tight in the absence of additional problem structure \citep[][Theorem 2.1]{wolkowicz1980bounds}.\qedhere
\end{proof}
Observe that Proposition \ref{propviol2}'s bound holds upon relaxing ${\mathrm{tr}(\bm{X})=T}$ to $\mathrm{tr}(\bm{X})\leq T$, as the later constraint is equivalent to writing $\mathrm{tr}(\bm{X})= t$ for some $t \in [0, T]$, and Proposition \ref{propviol2}'s bound is minimized over $t \in [0, T]$ by setting $t=T$.

Two important instances of Proposition \ref{propviol2}'s bound arise in correlation matrices where $\mathrm{Diag}(\bm{X})=\mathbf{e}$ and spectraplex matrices where $\mathrm{tr}(\bm{X})=1$. In these cases:
\begin{align*}
    \lambda_{\min}(\bm{X}) \geq 2-n, \ \text{and} \ \lambda_{\min}(\bm{X}) \geq \frac{2}{n}-1,
\end{align*}
respectively. Indeed, imposing the constraint $\mathrm{tr}(\bm{X})=1$ in Problem \eqref{primalsocprelaxation} provides a dimension-independent violation bound, while imposing the same constraint in Problem \eqref{ddsos} yields a bound of the order $\frac{1}{n}-n$.

The worst-case value of $\lambda_{\min}(\cdot)$ is $n$ times smaller in Problem \eqref{primalsocprelaxation} than Problem \eqref{ddsos}. Therefore, we have established that Problem \eqref{primalsocprelaxation}
is a tighter approximation than Problem \eqref{ddsos}. As optimizers widely believe that LOs and SOCPs are equally tractable in the age of modern IPMs, we will use Problem \eqref{primalsocprelaxation} as an initial approximation when developing a cutting-plane method in Section \ref{sec:outeraprox}, and not consider Problem \eqref{ddsos} for the rest of the paper.

\subsection{Two Classes of Separating Hyperplanes}\label{ssec:cuttingplanes}
We now present two classes of separating hyperplanes which allow us to improve the outer approximations presented in Problems \eqref{ddsos}-\eqref{primalsocprelaxation}. Our first class of separating hyperplanes rests on an almost tautological result, namely:
\begin{proposition}\label{rank1cut}
The following two constraints are equivalent:
\begin{align*}
    \bm{X} \succeq \bm{0}, \quad &\\
    \left\langle \bm{X}, \bm{Y} \right\rangle \geq 0, \quad & \forall \bm{Y} \in S_+^n: \Vert \bm{Y}\Vert_{*} \leq 1.
\end{align*}
\end{proposition}
\begin{proof}
The result follows from the self-duality of $S_+^n$.
\end{proof}

For a given $\bm{\hat{X}}$, Proposition \ref{rank1cut} yields an extreme ray ${\bm{Y}(\bm{\hat{X}}) \in S_+^n}$ in $O(n^2)$ time via the power method, see \cite{horn1990matrix}; namely $$\bm{Y}(\bm{\hat{X}})=\bm{y}\bm{y}^\top, \quad \text{where}  \quad \bm{y}= \arg \min_{\bm{x}: \Vert \bm{x}\Vert_2 =1}\bm{x}^\top \bm{\hat{X}} \bm{x}.$$ Moreover, $\bm{Y}$ either verifies membership of $\bm{\hat{X}} \in S_+^n$ if ${\left\langle \bm{\hat{X}}, \bm{Y} \right\rangle \geq \bm{0}}$, or provides a separating hyperplane such that $\left\langle \bm{\hat{X}}, \bm{Y}\right\rangle <0$ and $\left\langle \bm{X}, \bm{Y} \right\rangle \geq 0$ for any feasible $\bm{X} \in S_+^n$. We refer to this class of cuts as ``trailing eigenvalue cuts''.

The next class of separating hyperplanes is derived from a reformulation of the PSD cone which is, to our knowledge, new:
\begin{proposition}\label{rankkcut}
The following two constraints are equivalent:
\begin{align*}
    \bm{X} \succeq \bm{0}, \quad &\\
    \left\langle \bm{X}, \bm{Y}\right\rangle \leq \mathrm{tr}(\bm{X}), \quad & \forall \bm{Y} \in S_+^n: \Vert \bm{Y} \Vert_{\sigma} \leq 1.
\end{align*}
\end{proposition}
\begin{proof}
The first constraint is equivalent to requiring that \begin{align*}
    \Vert \bm{X}\Vert_{*} \leq \mathrm{tr}(\bm{X}),
\end{align*} because $\bm{X} \succeq \bm{0}$ if and only if $\lambda_{\min}(\bm{X}) \geq 0$. Moreover, all eigenvalues $\lambda_i(\bm{X})$ are non-negative if and only if the singular values $\sigma_i(\bm{X})$ are such that $\sum_i \sigma_i(\bm{X}) \leq \sum_i \lambda_i(\bm{X})$, since $\sigma_i(\bm{X}) = \vert \lambda_i(\bm{X}) \vert$ for any symmetric matrix $\bm{X}$; see \citep[2.6.P15]{horn1990matrix}. Thus, the result follows from the duality of the spectral and nuclear norms.
\end{proof}

The following proposition demonstrates that an extreme ray of Proposition \ref{rankkcut}'s reformulation of the PSD cone is given via a singular value decomposition (SVD):
\begin{proposition}\label{lemma:svdcuts}
Let $\bm{X} \in S^n$ be a symmetric matrix. Then, an optimal solution to the problem
\begin{align*}
    \max_{\bm{Y} \in S^n: \Vert \bm{Y} \Vert_\sigma \leq 1} \left\langle \bm{X}, \bm{Y} \right\rangle
\end{align*}
is given by setting $\bm{Y}^*:=\bm{U}\bm{U}^\top$, where $\bm{X}:=\bm{U}\bm{\Lambda}\bm{U}^\top$ is a singular value decomposition of $\bm{X}$.
\end{proposition}
\begin{proof}
Let $\lambda_i$ denote the eigenvalues of $\bm{X}$, and $\sigma_i=\vert \lambda_i\vert$ denote the singular values of $\bm{X}$. Then, the result follows from observing that $(1)$, by H{\"o}lder's inequality: $$\left\langle \bm{X}, \bm{Y} \right\rangle \leq \Vert \bm{X} \Vert_{*} \Vert \bm{Y} \Vert_\sigma, \qquad \forall \bm{X}, \bm{Y} \in S^n,$$ because the spectral norm is dual to the nuclear norm, and $(2)$ H{\"o}lder's inequality is tight when $\bm{Y}=\bm{Y}^*$, because $\bm{u}_i \perp \bm{u}_j, \ \forall i \neq j$ and therefore $\left\langle \bm{X}, \bm{Y}^* \right\rangle =\sum_{i=1}^n \vert \lambda_i \vert =\sum_{i=1}^n \sigma_i =\Vert \bm{X}\Vert_{*}$.%\footnote{Note that we can omit columns of $\bm{U}$ which have a singular value of $0$ without loss of generality.}..
\end{proof}
Observe that $\bm{Y}^*$ either verifies membership of $\bm{\hat{X}} \in S_+^n$ if $\left\langle \bm{\hat{X}}, \bm{Y} \right\rangle \leq \mathrm{tr}(\bm{\hat{X}})$, or provides a separating hyperplane such that $\left\langle \bm{X}, \bm{Y} \right\rangle > \mathrm{tr}(\bm{\hat{X}})$ and $\left\langle \bm{X}, \bm{Y} \right\rangle \leq \mathrm{tr}(\bm{X}), \ \forall \bm{X} \in S_+^n$. We refer to this class of separating hyperplanes as ``nuclear norm cuts''.

\subsubsection*{Pros and Cons of The Two Separation Oracles}
In this section, we restated a well-known separation oracle for the PSD cone, and derived a new separation oracle. Conceptually, both oracles are similar, as they both rest upon semi-infinite reformulations of the positive semidefinite cone. However, they admit practical differences. For a matrix $\bm{X}$ with $k$ negative eigenvalues, the trailing eigenvalue oracle provides $k$ orthogonal separating hyperplanes, while the nuclear norm oracle only provides one hyperplane. Moreover, as observed by \citet{permenter2014partial}, $k$ trailing-eigenvalue cuts can be aggregated into a single nonlinear cut, which ensures that $\bm{X}$ is positive semidefinite on the subspace spanned by all $k$ trailing eigenvectors, by imposing a $k \times k$ semidefinite constraint \cite{permenter2014partial}.

In spite of the aggregate power of trailing-eigenvalue cuts, nuclear norm cuts may sometimes be preferable. Indeed, while imposing trailing-eigenvalue cuts sometimes causes the multiplicity of $\bm{X}$'s smallest eigenvalue to increase \cite{sivaramakrishnan2002linear}, nuclear norm cuts penalize the absolute sum of all $n$ eigenvalues, rather than the worst-case eigenvalue, and therefore do not have this defect. Moreover, as we show in Section \ref{sec:numnormmin}, nuclear norm cuts minimize the nuclear norm of a matrix in the problem's original space, while trailing eigenvalue cuts can only be applied to nuclear norm minimization after reformulating the problem in a higher-dimensional space, which is inefficient. We summarize the advantages and disadvantages of both cuts in Table \ref{tab:comparison_hyperplanes}.
\begin{table}[h]\small
\centering
\caption{Summary of advantages ($+$)/ disadvantages ($-$) of the cuts.}
\begin{tabular}{ll} \toprule
 Cut & Characteristics \\ \midrule
\multirow{4}{*}{Trailing EV} & ($+$) Each cut generated in $\Omega(n^2)$ time \\
& ($+$) Multiple cuts per iterate\\
& ($+$) Cuts can be aggregated  \\
& ($-$) Aggregating cuts is expensive  \\
%& ($-$) Trailing eigenvalues may cluster \\
& ($-$) Models nuclear norm in extended space\\
\midrule
\multirow{2}{*}{Nuclear Norm} & ($+$) Penalizes trailing eigenvalue clustering\\
& ($+$) Models nuclear norm in original space\\
& ($-$) Each cut generated in $\Omega(n^3)$ time \\
\bottomrule
\end{tabular}
\label{tab:comparison_hyperplanes}
\end{table}
\section{An Efficient Algorithmic Approach}\label{sec:outeraprox}
In this section, we present an efficient numerical approach to solve Problem \eqref{primalsdp}. Our approach is an outer-approximation strategy, which takes a tractable SOC-representable approximation of the original problem, and iteratively improves the approximation by adding cutting planes.
\subsection{A Cutting-Plane Method}
The previous section presents tractable outer-approximations of Problem \eqref{primalsdp} and demonstrates that these approximations can be iteratively improved by adding cutting planes. Consequently, a valid numerical strategy for solving Problem \eqref{primalsdp} is the cutting-plane method, which was originally proposed by \citet{kelley1960cutting} under an assumption that the problem's feasible region is polyhedral. The cutting-plane method replaces the constraint $\bm{X} \succeq \bm{0}$ with infinitely many linear constraints of the form $f(\bm{X}, \bm{Y}) \leq 0$, which must hold for each $\bm{Y}$ in some convex set $\mathcal{Y}$. Specifically,
\begin{gather}
    f(\bm{X}, \bm{Y})=-\left\langle \bm{X}, \bm{Y} \right\rangle \leq 0, \forall \bm{Y} \in \mathcal{Y}=\left\{\bm{Y} \in S_+^n: \Vert \bm{Y} \Vert_* \leq 1\right\}
\end{gather}
for trailing eigenvalue cuts and
\begin{gather}
    f(\bm{X}, \bm{Y})=\left\langle \bm{X}, \bm{Y}-\mathbb{I} \right\rangle \leq \bm{0}, \forall \bm{Y} \in \mathcal{Y}=\{\bm{Y} \in S_+^n: \Vert \bm{Y} \Vert_\sigma \leq 1\}
\end{gather} for nuclear-norm cuts. After performing this reformulation, the cutting-plane method iteratively minimizes over outer-approximations of Problem \eqref{primalsdp}, which are given by imposing finitely many constraints of the form $f(\bm{X}, \bm{Y}^t) \leq 0, \forall t \in [T].$

These approximations are iteratively improved by (a) minimizing over them to obtain optimal but infeasible iterates $\bm{X}^t$, and (b) refining them by adding a constraint corresponding to a most violated $\bm{Y}$ at $\bm{X}^t$, $\bm{Y}^t$, where $$\bm{Y}^t=\arg\max_{\bm{Y} \in \mathcal{Y}} f(\bm{X}^t, \bm{Y}).$$

This numerical strategy yields a sequence of optimal but infeasible iterates, until some iterate $\bm{X}^t$ is $\epsilon$-feasible, i.e., $$f(\bm{X}^t, \bm{Y}^t) \leq \epsilon$$ at which point the method terminates with an optimal and $\epsilon$-feasible solution. We present this strategy in Algorithm \ref{alg:OA}.

\begin{algorithm}
\caption{Outer-approximation scheme}
\label{alg:OA}
\begin{algorithmic}
\REQUIRE Feasibility tolerance $\epsilon$
\STATE $t \leftarrow 1 $
\REPEAT
\STATE Compute $\bm{X}^{t+1}$ solution of
\begin{align*}
\min \quad &  \left\langle \bm{C}, \bm{X} \right\rangle \\ \text{s.t.} \quad & \left\langle \bm{A}_i, \bm{X} \right\rangle =b_i, \ \forall i \in [m], \ f(\bm{X}, \bm{Y}^i) \leq 0, \ \forall i \in [t-1],\\
& \bigg\vert \bigg\vert \begin{pmatrix}
    2X_{i,j}\\
    X_{i,i}-X_{j,j}
    \end{pmatrix} \bigg\vert \bigg\vert_2 \leq X_{i,i}+X_{j,j}, \ \forall i \in [n], \ \forall j \in [n].
\end{align*}
\STATE Compute $\bm{Y}^{t+1}$ solution of $\max_{\bm {Y} \in \mathcal{Y}} f(\bm{X}^t, \bm{Y})$.
\STATE Impose constraint $ f(\bm{X}, \bm{Y}^{t+1}) \leq 0$.
\STATE $t \leftarrow t+1 $
\UNTIL{$f(\bm{X}^{t}, \bm{Y}^{t}) \leq \epsilon$}
\RETURN $\bm{X}^{t}$
\end{algorithmic}
\end{algorithm}

We now prove that Algorithm \ref{alg:OA} converges. Proofs of convergence for cutting-plane methods rely on compactness arguments, which can be established by bounding Problem \eqref{primalsocprelaxation}'s feasible region. To supply a proof of convergence, we require:

\begin{assumption}\label{assumption:boundedness}
All feasible solutions to Problem \eqref{primalsdp} are contained within a convex set $\mathcal{X}$, such that $\mathrm{tr}(\bm{X}) \leq T, \forall \bm{X} \in \mathcal{X}$.
\end{assumption}
Assumption \ref{assumption:boundedness} is often satisfied naturally. For instance, the trace of $\bm{X}$ is constrained in SDO relaxations of binary quadratic and sparse principal component analysis problems. Alternatively, Assumption \ref{assumption:boundedness} can be satisfied by imposing the big-M constraint $\mathrm{tr}(\bm{X}) \leq M$ for a large positive number $M$.

Our proof of convergence also rests on the requirement that the functions $f(\bm{X}, \bm{Y}^{t})$ are Lipschitz continuous, i.e., there exists a Lipschitz constant $L$ such that for any feasible $\bm{X}_1, \bm{X}_2$ in Problem \eqref{primalsocprelaxation} and any separating hyperplane $\bm{Y}$, $$\vert f(\bm{X}_1, \bm{Y})-f(\bm{X}_2, \bm{Y})\vert \leq L \Vert \bm{X}_1-\bm{X}_2 \Vert_F.$$ This holds for both cuts outlined previously, by combining the Cauchy-Schwarz inequality with norm equivalence. Indeed, $$\vert f(\bm{X}_1, \bm{Y})-f(\bm{X}_2, \bm{Y})\vert \leq \Vert \bm{X}_1-\bm{X}_2\Vert_F \Vert \bm{Y}\Vert_F=\Vert \bm{X}_1-\bm{X}_2\Vert_F$$ for trailing eigenvalue cuts and $$\vert f(\bm{X}_1, \bm{Y})-f(\bm{X}_2, \bm{Y})\vert \leq \Vert \bm{X}_1-\bm{X}_2\Vert_F \Vert \bm{Y}-\mathbb{I}\Vert_F \leq 2\sqrt{n} \Vert \bm{X}_1-\bm{X}_2\Vert_F$$ for nuclear norm cuts. We now state our main result:
\begin{theorem}\label{proofofconvergence}
Suppose that Assumption \ref{assumption:boundedness} holds and the separation oracle in Algorithm \ref{alg:OA} is exact. Let $\bm{X}^t$ be a feasible solution returned by the $t$-th iterate of Algorithm \ref{alg:OA}, where
\begin{align*}
    t \geq \left(\frac{L T}{\epsilon}+1\right)^{n^2}
\end{align*} $L$ is the Lipschitz constant and $T$ is the aforementioned trace bound. Then, $\bm{X}^t$ is an optimal and $\epsilon$-feasible solution to Problem \eqref{primalsdp}.  Moreover, any limit point of $\{\bm{X}_t\}_{t=1}^\infty$ solves \eqref{primalsdp}.
\end{theorem} %Todo: check n^2 vs. n here.
\begin{remark}
Note that the sequence of iterates $\{\bm{X}^t\}_{t=1}^\infty$ need not converge to a single solution $\bm{X}^*$. For instance, there could be multiple optimal solutions and multiple subsequences of $\left\{\bm{X}^t\right\}_{t=1}^\infty$ which each converge towards different minimizers. However, if Problem \eqref{primalsdp} admits a unique optimal solution and Assumption \ref{assumption:boundedness} holds then $\{\bm{X}^t\}_{t=1}^\infty$ converges to the unique optimal $\bm{X}^*$.
\end{remark}
\begin{remark}\label{remark:unique}
An alternative to Assumption \ref{assumption:boundedness} is to regularize the objective with a strongly convex term $\frac{1}{\gamma}\Vert \bm{X}\Vert_F^2$ for some $\gamma >0$. This term ensures the boundedness of a cutting-plane method's iterates \textit{without Assumption
\ref{assumption:boundedness}}, by making the objective function coercive.
%\footnote{Given a feasible solution $\hat{\bm{X}}$ with objective cost $\theta$, imposing the constraint $\left\langle \bm{C}, \bm{X} \right\rangle+\frac{1}{\gamma}\Vert \bm{X}\Vert_F^2 \leq \theta$ transforms the problem into an equivalent bounded problem with the same optimal solution.}
Moreover, it guarantees that Algorithm \ref{alg:OA} converges to the unique optimal $\bm{X}^*$, because the regularized problem admits a unique optimal solution \cite[see][Proposition 1.1.2]{bertsekas1999nonlinear}.
\end{remark}

Proofs of convergence of general cutting-plane methods have been supplied by a number of authors since the work of \citet{kelley1960cutting}; see for instance \citep[Section 7.5.3]{bertsekas1999nonlinear}, \citep[Section 5.2]{mutapcic2009cutting}. To keep this paper self-contained, and make explicit Theorem \ref{proofofconvergence}'s dependence upon Assumption \ref{assumption:boundedness}, we now supply a proof of convergence; essentially due to \cite{mutapcic2009cutting}:
\begin{proof}
Suppose that at some iteration $k>1$, Algorithm \ref{alg:OA} has not converged. Then, we have that $f(\bm{X}^k, \bm{Y}^k) > \epsilon$, but $f(\bm{X}^k, \bm{Y}^i) \leq 0, \forall i < k$, which implies that $$\epsilon < f(\bm{X}^k, \bm{Y}^k)-f(\bm{X}^i, \bm{Y}^k) \leq L \Vert \bm{X}^i-\bm{X}^k\Vert_F,$$ where the second inequality follows by Lipschitz continuity.
%, $L=1$ for trailing eigenvalue cuts and $L \leq 2\sqrt{n}$ for nuclear norm cuts.
Rearranging the above inequality implies $\Vert \bm{X}^i-\bm{X}^k\Vert_F > \frac{\epsilon}{L}$, i.e., Algorithm \ref{alg:OA} never visits any ball of radius $\frac{\epsilon}{L}$ twice. Moreover, by iteration $k$, Algorithm \ref{alg:OA} has visited non-overlapping balls with combined volume$$k \frac{\pi^\frac{n^2}{2}}{\Gamma(\frac{n^2}{2}+1)} \left(\frac{\epsilon}{L}\right)^{n^2}, $$
and these balls must be centered at feasible points, i.e., contained within a ball of radius $T+\frac{\epsilon}{L}$ which has volume $$ \frac{\pi^\frac{n^2}{2}}{\Gamma(\frac{n^2}{2}+1)} \left(T+\frac{\epsilon}{L}\right)^{n^2}. $$ That is, if Algorithm \ref{alg:OA} has not converged at iteration $k$, we have:
\begin{align*}
    k < \left(\frac{LT}{\epsilon}+1\right)^{n^2},
\end{align*}
which implies that we converge to an $\epsilon$-optimal solution within $k\leq \big(\frac{LT}{\epsilon}+1\big)^{n^2}$ iterations, for any $\epsilon>0$.
\end{proof}

We note that Theorem \ref{proofofconvergence}'s bound depends explicitly on Assumption \ref{assumption:boundedness}'s trace bound. Therefore, Algorithm \ref{alg:OA} is particularly well-suited to problems with a constraint on $\mathrm{tr}(\bm{X})$. Moreover, the above bound illustrates why a SOC-representable initial feasible region will outperform an LO-representable initial feasible region in practice: a worst-case rate of convergence which depends on $\mathrm{tr}(\bm{X})$, rather than $n \cdot \mathrm{tr}(\bm{X})$, often translates to numerically superior performance in practice.

We close this section by remarking that while IPMs require that strong duality holds, Algorithm \ref{alg:OA} does not. Consequently, Algorithm \ref{alg:OA} may be a good method for SDOs which IPMs cannot solve because of constraint qualification failures.
\section{Numerical Results}\label{sec:numres}
In this section, we apply Algorithm \ref{alg:OA} to large-scale sparse principal component analysis and nuclear norm minimization problems, using trailing eigenvalue (resp. nuclear norm) cuts for sparse PCA (resp. nuclear norm minimization). All experiments were implemented in \verb|Julia| $1.0$
using \verb|MOSEK| $9.0$ and \verb|JuMP.jl| $0.18.5$ \citep{dunning2017jump}, and performed on one Intel Xeon E$5$-$2690$ v4 $2.6$GHz CPU, using $32$ GB RAM and six CPU cores with two threads each ($12$ threads total).
\subsection{Sparse Principal Component Analysis}
Given high-dimensional data $\bm{A} \in \mathbb{R}^{n \times p}$ and its normalized centered covariance matrix $\bm{\Sigma}:=\frac{1}{p-1}\bm{A}^\top \bm{A} \in S^n$, a common task is to compress $\bm{A}$, by projecting $\bm{A}$ onto a small number of principal components. This task is known as principal component analysis (PCA), and is achieved by performing singular value decomposition to obtain $\bm{\Sigma}=\bm{S}\bm{\Lambda}\bm{S}^\top$, and projecting $\bm{A}$ onto the leading eigenvectors via $\bm{A}_{\text{new}}=\bm{S}_{[1:k]}\bm{A}$.

PCA is sometimes criticized for producing uninterpretable features, because each new feature is a linear combination of all $n$ original features. This is often unacceptable because:
\begin{itemize}  \setlength\itemsep{0em}
    \item In medical applications, downstream decisions taken via PCA must be interpretable.
    \item In financial applications such as investing across index funds, each non-zero entry in each PC incurs a cost.
\end{itemize}

One approach to obtain interpretable principal components is to stipulate that they are sparse, i.e., have at most $k$ non-zero entries. This approach leads to the following problem \citep[see][]{d2005direct}:
\begin{align}\label{OriginalSPCA}
    \lambda_{\max}^k(\bm{\Sigma}):=\max_{\bm{x} \in \mathbb{R}^n} \ \bm{x}^\top \bm{\Sigma} \bm{x} \
    \text{s.t.} \ \bm{x}^\top \bm{x}\leq 1,\ \vert \vert \bm{x} \vert \vert_0 \leq k.
\end{align}
%where the sparsity constraint $\vert \vert \bm{x} \vert \vert_0 \leq k$ forces variance to be explained in a compelling manner.
Unfortunately, this problem cannot be solved to certifiable optimality when $n>100$  \citep[see][for a certifiably optimal approach]{berk2018certifiably}.

Motivated by the observation that SDO relaxations provide near-exact upper bounds on $\lambda_{\max}^k(\bm{\Sigma})$ in practice,
\citet{d2005direct} proposed the relaxation:
\begin{align}\label{robustevproblem}
    \max_{\bm{X} \in S_+^n} \quad & \left\langle \bm{\Sigma}, \bm{X} \right\rangle \ \text{s.t.} \ \mathrm{tr}(\bm{X})=1, \Vert \bm{X} \Vert_1 \leq k.
\end{align}

Unfortunately, modern IPMs cannot solve Problem \eqref{robustevproblem} for ${n>250}$, because all known reformulations of the constraint $\Vert \bm{X} \Vert_1 \leq k$ require $\Omega(n^2)$ linear inequality constraints. Therefore, we apply Algorithm \ref{alg:OA} (with trailing eigenvalue cuts) to this problem, and compare the bounds obtained via Algorithm \ref{alg:OA} to the best feasible solution obtained by the method of \citet{berk2018certifiably} (with a time limit of $60$ seconds). We also compare our bounds with the exact SDO bound given in Problem \eqref{robustevproblem}, wherever it can be computed by \verb|MOSEK| within our $32$GB peak memory budget.

For completeness, we now describe the datasets which we benchmark our approach on. Note that the first three datasets are distributed via the UCI Machine Learning Repository \cite{bache2013uci}.
%\footnote{\href{https://archive.ics.uci.edu/ml/datasets.html}{https://archive.ics.uci.edu/ml/datasets.html.}}:

\begin{itemize}
\item The Pitprops dataset: a $13 \times 13$ covariance matrix derived from $180$ observations of $13$ features.
\item The normalized communities dataset: a $101 \times 101$ correlation matrix derived from $1994$ observations of $128$ features, after eliminating $27$ categorical variables.
\item The normalized Arrhythmia dataset: a $274 \times 274$ correlation matrix derived from $452$ observations of $279$ features, after eliminating $5$ categorical variables.
\item Wilshire $5000$: a $2130 \times 2130$ correlation matrix derived from
$2768$ observations of changes in daily stock prices from $3$ January $2007$ to $31$ December $2017$, obtained via Yahoo! Finance using the R package \textit{quantmod} \citep[see][]{ryan2018quantmod}, after eliminating stocks which were not traded in all $2768$ days considered.
\end{itemize}
As optimizing over $n^2$ SOC constraints is too expensive when $n >2000$, we adopt a more tractable initial approximation for the Wilshire $5000$ dataset. Namely, we aggregate the $n^2$ SOC constraints into $n$ constraints, by (a) observing that
\begin{align*}
    X_{i,i} X_{j,j} \geq X_{i,j}^2 \implies X_{i,i} \geq \sum_{j=1}^n X_{i,j}^2,
\end{align*} since $\mathrm{tr}(\bm{X})=1$, and (b) introducing the auxiliary variables $\bm{z}$ to model a relaxation of the sparsity constraint $\Vert \bm{x}\Vert_0 \leq k$. This yields the following SOC relaxation, which is weaker than Problem \eqref{primalsocprelaxation}'s relaxation, but has $O(n)$, rather than $O(n^2)$, SOC/linear inequality constraints:
    \begin{align*}\label{primalsocprelaxation2}
        \min_{\bm{X} \in S^n, \bm{z} \in \mathbb{R}_+^n} \ \left\langle \bm{\Sigma}, \bm{X} \right\rangle \ \text{s.t.} \ \mathrm{tr}(\bm{X}) =1, \bm{z} \leq \mathbf{e}, \bm{e}^\top \bm{z} \leq k, \sum_{j=1}^n X_{i,j}^2 \leq z_i X_{i,i}, \forall i.
    \end{align*}

We now present our experimental results. Table \ref{tab:spcacomp1} depicts the time required for Algorithm \ref{primalsdp} to generate a bound with $0, 5$ or $20$ cuts, and the time required for \verb|MOSEK| to solve the SDO relaxation when $k=10$; Table \ref{tab:spcacomp2} depicts the magnitude of the bound gaps for the different bounds when $k=10$; Figure \ref{fig:spcafig} depicts the quality of the SOC and SDO bounds, for varying $k$, against the optimal solution obtained by the algorithm of \cite{berk2018certifiably} (for $n \leq 101$) or the best solution obtained after $60$ seconds (for $n > 101$). We do not generate cuts for the Wilshire $5000$ dataset, as the SOC bound is sufficiently tight here.

\begin{table}[h]\footnotesize
\caption{Runtime (s) for $k=10$ (resp. $k=50$) for UCI (Wilshire) datasets.}
\centering
\begin{tabular}{l r r r r}\toprule \label{tab:spcacomp1}
 Problem &  & \multicolumn{3}{c}{Approach runtime (s)} \\ \cmidrule{2-5}
  & SOC & SOC $5$ cuts & SOC $20$ cuts & SDO \\ \midrule
pitprops & $0.1$ & $0.2$& $0.7$ & $0.1$\\
norm communities & $2.8$ & $9.7$ & $26.3$ & $233.6$\\
norm arrythmia & $37.5$ & $188.6$ & $564.1$ & n/a \\
norm wilshire & $870.6$ & n/a & n/a & n/a \\
     \bottomrule
\end{tabular}
\end{table}
\begin{table}[h]\footnotesize
\caption{Bound gap ($\%$) for $k=10$ (resp. $k=50$) for UCI (Wilshire) datasets.}
\centering
\begin{tabular}{l r r r r}\toprule \label{tab:spcacomp2}
 Problem  & \multicolumn{4}{c}{Bound gap ($\%$) } \\ \cmidrule{2-5}
 & SOC & SOC $5$ cuts & SOC $20$ cuts & SDO \\ \midrule
    pitprops &  $6.60$  & $2.10$  & $1.11$  & $1.09$ \\
     norm communities & $0.83$  & $0.75$  & $0.75$  & $0.75$ \\
     norm arrythmia & $3.01$  & $2.16$  & $1.54$  & n/a \\
    norm wilshire & $0.48$  & n/a & n/a & n/a \\
     \bottomrule
\end{tabular}
\end{table}

\begin{figure*}[h]
    \centering
    \begin{subfigure}[t]{0.475\textwidth}
        \centering
\includegraphics[scale=0.3]{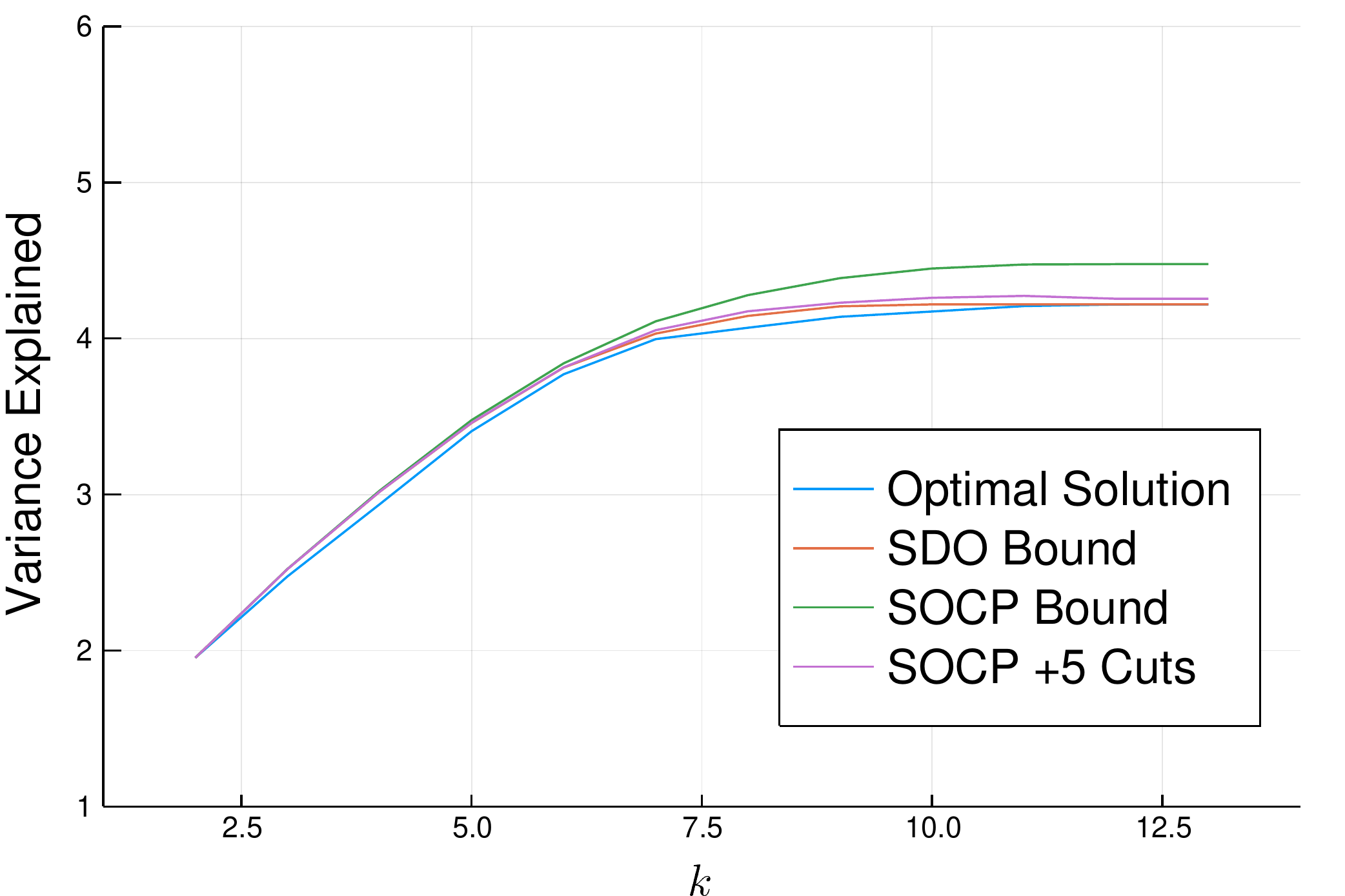}
    \caption{Pitprops data set ($n=13)$.}
    \end{subfigure}%
    \hfill
    \begin{subfigure}[t]{0.475\textwidth}
        \centering
\includegraphics[scale=0.3]{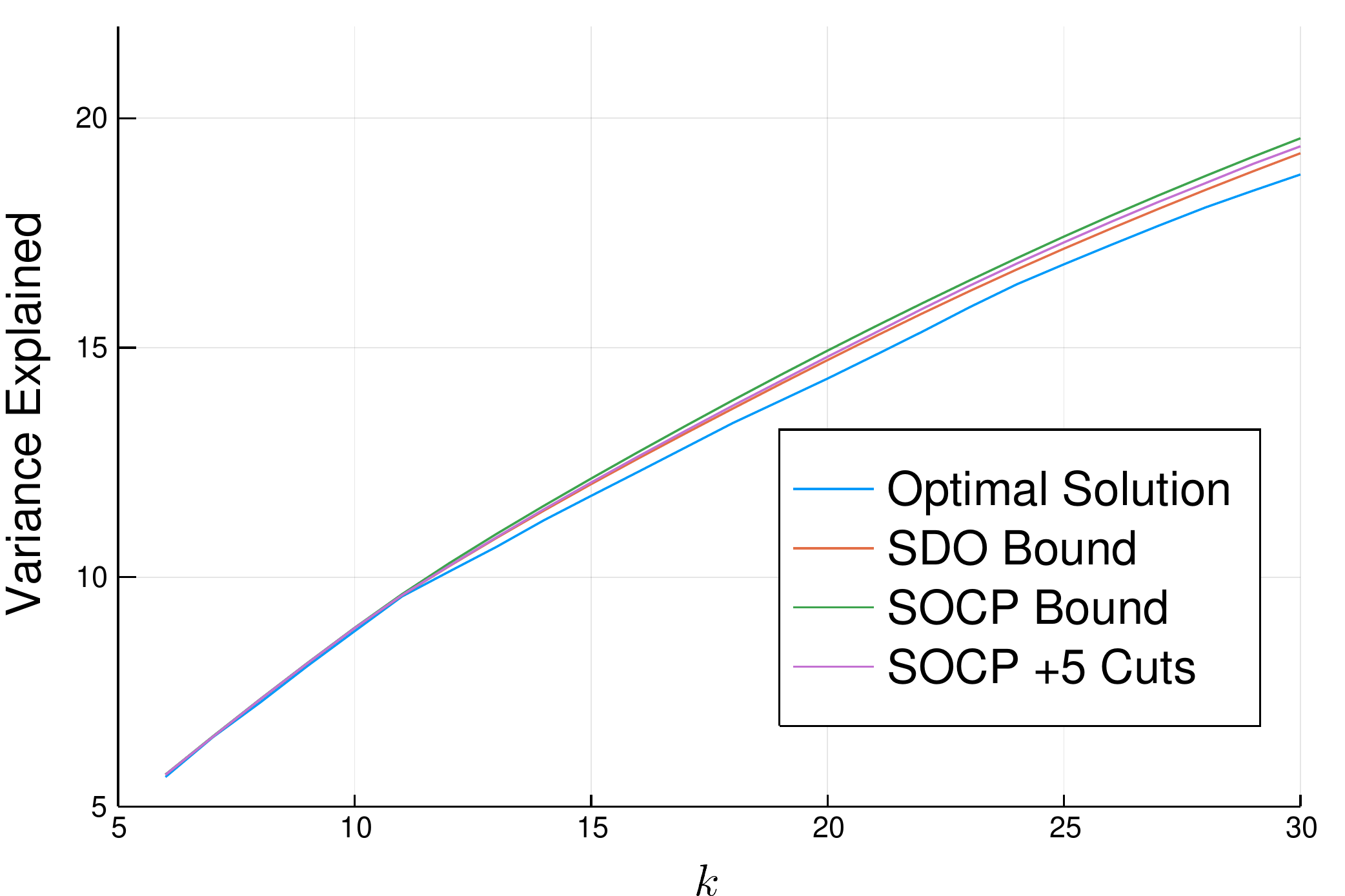}
\caption{Normalized communities data set ($n=101$).}
    \end{subfigure}
    \vskip\baselineskip
\begin{subfigure}[b]{0.475\textwidth}
            \centering
\includegraphics[scale=0.3]{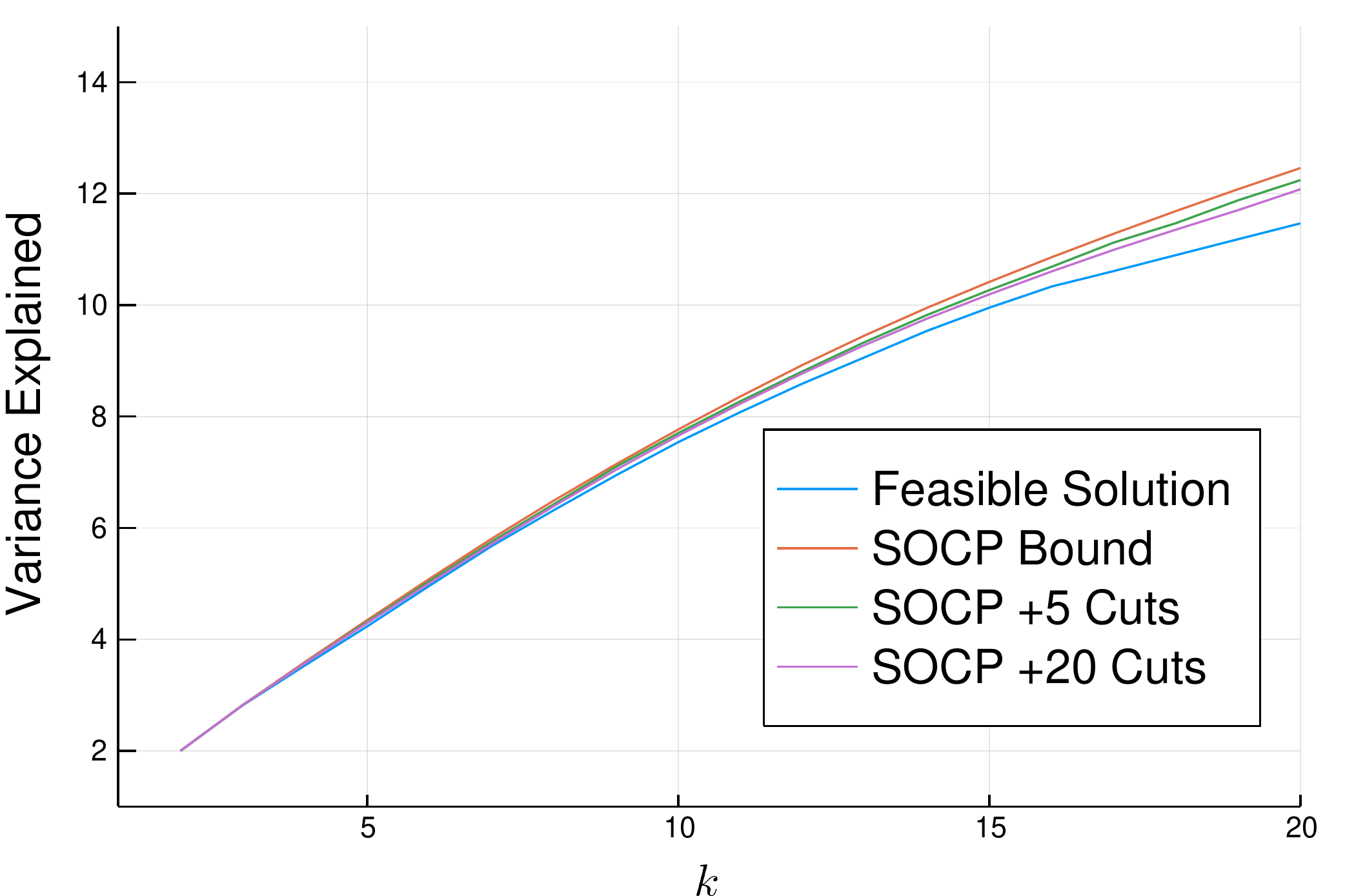}
\caption{Normalized arrhythmia data set ($n=274$).}
\end{subfigure}
\hfill
\begin{subfigure}[b]{0.475\textwidth}   \centering
\includegraphics[scale=0.3]{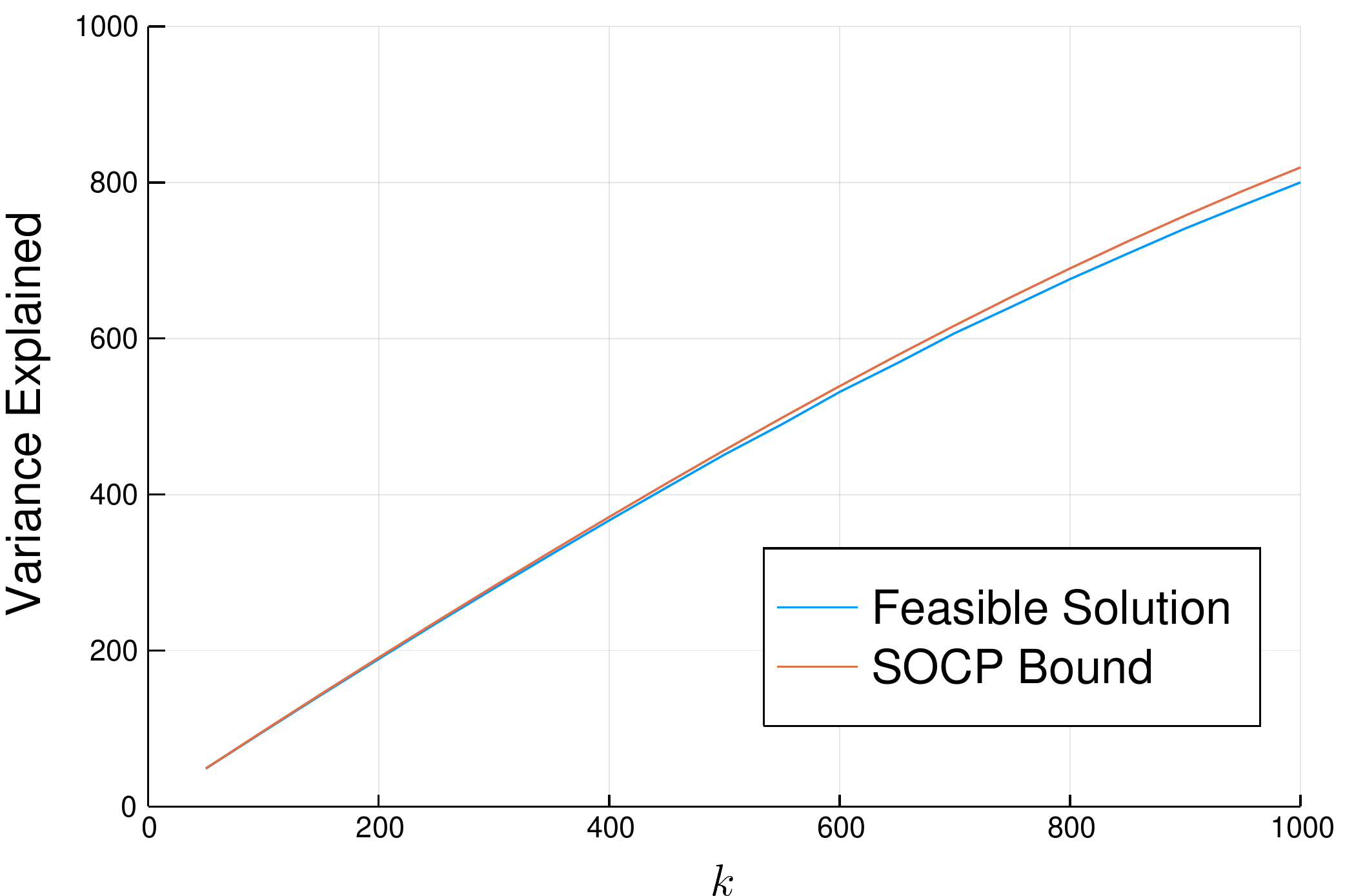}
\caption{Normalized Wilshire data set ($n=2130$).}
\end{subfigure}
\caption{Quality of SOC and SDO bounds on the leading principal component for different UCI datasets.} \label{fig:spcafig}
\end{figure*}

Our main findings from this set of experiments are as follows:
\begin{itemize}
    \item Algorithm \ref{alg:OA} provides bounds for sparse PCA problems which are as tight as those given by the SDO relaxation of \cite{d2005direct}, yet scale to problems an order of magnitude larger.
    \item For sparse PCA problems where $n \leq 1000$, Problem \eqref{primalsocprelaxation}'s outer approximation provides an accurate upper bound in a tractable fashion. Moreover, for larger problems where $1000 <n  \leq 10000$, aggregating the SOCP constraints yields a looser yet still tractable bound. This is because Problem \eqref{robustevproblem} contains the constraint $\mathrm{tr}(\bm{X})=1$, which, as established in Proposition \ref{propviol2}, provides tight constraint violation guarantees.
\end{itemize}

\subsection{Nuclear Norm Minimization}\label{sec:numnormmin}
Given an incomplete set of observations $A_{(i,j)} \  \forall (i,j) \in \Omega$ of a low-rank matrix $\bm{A} \in S^n$, a central problem in machine learning is to recover the entire matrix, by solving:
\begin{equation}\label{rankminprob}
\begin{aligned}
    \min \ \mathrm{Rank}(\bm{X}) \ \text{s.t.} \ A_{(i,j)}=X_{(i,j)}, \ \forall (i,j) \in \Omega.
\end{aligned}
\end{equation}

Unfortunately, all known algorithms for solving Problem \eqref{rankminprob} to certifiable optimality require
%doubly
exponential time
%\footnote{Problem \eqref{rankminprob} could be solved to certifiable optimality by rewriting the rank objective as exponentially many semi-algebraic constraints, and invoking the \citet{lasserre2001global} hierarchy, as proposed by \citet{d2003semidefinite}.}
in both theory and practice \cite{recht2010guaranteed}. Consequently, various authors including \cite{ recht2010guaranteed} have proposed instead solving Problem \eqref{rankminprob}'s convex relaxation:%, which is:
\begin{equation}\label{nucminprob}
\begin{aligned}
    \min \ \Vert\bm{X}\Vert_{*} \ \text{s.t.} \ A_{(i,j)}=X_{(i,j)}, \ \forall (i,j) \in \Omega.
\end{aligned}
\end{equation}
By exploiting Schur complements, \citet{recht2010guaranteed} have established that Problem \eqref{nucminprob} is SDO-representable. Unfortunately, their reformulation is intractable for even medium-sized problems, as it rests upon lifting the nuclear norm objective to a higher dimensional space. An alternative approach is to apply Algorithm \ref{alg:OA} with nuclear norm cuts, i.e., exploit Proposition \ref{rankkcut}'s semi-infinite reformulation of the nuclear norm objective, and iteratively solve a sequence of SOCPs of the form:
\begin{align*}
    \min \quad & \theta \quad \text{s.t.} \quad \theta \geq \Vert \bm{X}\Vert_F,\\
    &  \theta \geq \left\langle \bm{X}, \bm{Y} \right\rangle, \quad \forall \bm{Y} \in \bar{\mathcal{Y}},\\
     & A_{(i,j)}=X_{(i,j)}, \quad \forall (i,j) \in \Omega,
\end{align*}
where we impose the SOC-representable Frobenius norm constraint, because it is valid by the relation $\Vert \bm{X} \Vert_F \leq \Vert \bm{X} \Vert_{*}$ \citep[see][]{horn1990matrix}. Moreover, Proposition \ref{lemma:svdcuts} proves that at each iteration $t$ a most-violated $\bm{Y}^t$ in the semi-infinite constraint $\theta \geq \left\langle \bm{X}, \bm{Y}\right\rangle, \ \forall \bm{Y} \in \mathcal{Y}$ is given by taking an SVD of $\bm{X}^t=\bm{U}\bm{\Sigma}\bm{U}^\top$ and setting $\bm{Y}=\bm{U}\bm{U}^\top$.

When we performed our numerical experiments, rather than solving Problem \eqref{nucminprob}, we applied Algorithm \ref{alg:OA} to the following problem (taking $\gamma=\frac{1}{n}$ in our experiments).
\begin{equation}\label{nucminprob_ridge}
\begin{aligned}
    \min \quad & \Vert\bm{X}\Vert_{*}+\frac{1}{\gamma}\Vert \bm{X}\Vert_F^2\\
    \text{s.t.} \quad & A_{(i,j)}=X_{(i,j)}, \quad \forall (i,j) \in \Omega.
\end{aligned}
\end{equation}
Solving this problem as a surrogate for Problem \eqref{nucminprob} serves a dual purpose: (a) it ensures that Problem \eqref{nucminprob} has a unique solution, which prevents degeneracy and discourages stalling behaviour, and (b) it ensures that $\{\bm{X}_t\}_{t=1}^\infty$ converges to Problem \eqref{nucminprob_ridge}'s unique solution; see Remark \ref{remark:unique}.

We now compare the time required for our approach to obtain a solution within $0.1\%$ of optimality, to the time required for \verb|MOSEK| to solve a semidefinite reformulation of the problem, where we generate a rank$-10$ matrix $R$ from two factors $U,V$ with i.i.d. $\mathcal{N}(0,1)$ entries, and randomly sample half the entries from the matrix, i.e., let $\vert\Omega\vert=\frac{n^2}{2}$. {\color{black}As \verb|MOSEK| does not allow SOC constraints and quadratic objective terms to be mixed, we model the Frobenius norm term by invoking the equivalence $$t \geq \Vert \bm{x}\Vert_2^2 \iff t+1\geq \left\Vert \begin{pmatrix} 2 \bm{x}\\t-1\end{pmatrix}\right\Vert_2.$$ }Figure \ref{fig:my_label} depicts the runtime requirement for both approaches (averaged over $5$ instances), and demonstrates that Algorithm \ref{alg:OA} outperforms IPMs when ${n \geq 100}$ (i.e., there are $\geq 5000$ constraints).

\begin{figure}[h]
    \centering
    \begin{tikzpicture}[scale=0.6]
      \begin{semilogyaxis}[
        ybar,
        xlabel={n},
        ylabel={Runtime (s)},
        width=0.65\textwidth,
        height=3cm,
        legend style={at={(1.05,0.75)},anchor=west},
        xtick={1, 2, 3, 4, 5, 6, 7},
        xticklabels={$50$, $100$, $150$, $200$, $250$, $300$, $350$},
      ]
      \addplot plot coordinates {(1, 4.56) (2, 32.03) (3, 77.96) (4, 233.4) (5, 677.3) (6, 1059.6) (7, 1461.372)};
      \addplot plot coordinates {(1, 2.7) (2, 68.7) (3, 570.5) (4, 2679.8)};
      \legend{OA, MOSEK}
      \end{semilogyaxis}
    \end{tikzpicture}
    \caption{Runtime requirement for Algorithm \ref{alg:OA} with nuclear norm cuts (``OA'') vs MOSEK applied to \cite{recht2010guaranteed}'s reformulation. Note that MOSEK cannot solve instances of \cite{recht2010guaranteed}'s SDO-representable reformulation with $n\geq 250$, as this requires a peak memory budget $>32$ GB RAM.}
    \label{fig:my_label}
\end{figure}
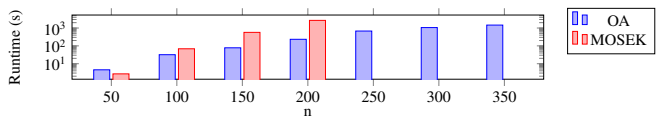

Our approach also generates solutions to Problem \eqref{nucminprob_ridge} for larger problem sizes. Figure \ref{fig:my_label2} depicts the convergence profile of our approach for an instance (generated in the same manner as the previous experiment) where $n=500$; convergence was attained after $1245$ seconds. The convergence profile is typical: we obtain a near-exact lower bound within $10$ iterations, and require many more iterations to identify the optimal solution.

\begin{figure}[h]
    \centering
    \includegraphics[scale=0.75]{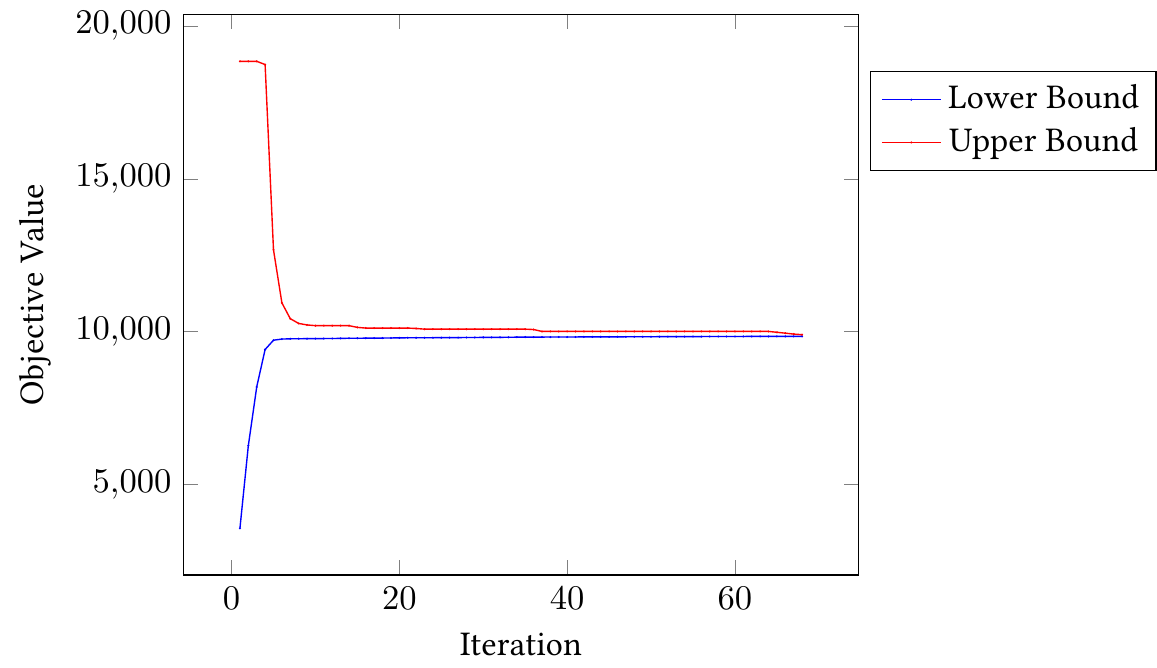}
    \caption{A convergence profile of Algorithm \ref{alg:OA} for $n=500$. The upper bound is given by the smallest value of $\Vert\bm{X}^t\Vert_{*}+\frac{1}{\gamma}\Vert \bm{X}^t \Vert_{F}^2$ observed by iteration $t$.}
    \label{fig:my_label2}
\end{figure}

Our main set of findings from this experiment are as follows:
\begin{itemize}
\item Although IPMs scale better than outer-approximation methods from a traditional complexity theory perspective, in practice Kelley's cutting-plane method solves nuclear norm minimization problems faster. This is because (a) Kelley's method solves nuclear norm minimization problems in their original space, while an SDO reformulation lifts the problem via Schur complements, and (b) while the traditional complexity theory accurately describes the performance of IPMs, it is too conservative with respect to cutting-plane methods. %Indeed, Kelley's cutting-plane method requires substantially fewer iterations than the number suggested by Theorem \ref{proofofconvergence}'s worst-case bound.
\item The bottleneck in applying Algorithm \ref{alg:OA} to larger-scale matrix completion problems is the SVD step, which scales at a rate of $\Omega(n^3)$. Therefore, one future direction is to replace the SVD step with an inexact oracle, which performs matrix sketching to yield valid inexact cuts in $\Omega(n^2)$ time.
\end{itemize}

\section*{Acknowledgements}
We are grateful to Peter Cohen for editorial comments.
% \section*{References}

\end{document}